\documentclass[11pt,oneside,a4paper]{amsart}
\usepackage[utf8]{inputenc}
\usepackage[margin=1.5in]{geometry}
\usepackage{latexsym,amssymb,upref,amsmath,amssymb,amsfonts,amsthm}
\usepackage{graphicx}
\usepackage[backend=biber, style=numeric, sorting=anyt, maxnames=10, maxalphanames=10, natbib=true, arxiv=abs, doi=false, url=false, eprint=true]{biblatex}
\usepackage{upref}
\usepackage{booktabs}
\usepackage{hyperref}

\newcommand{\articletitle}{Terminal-Pairability in Complete Bipartite Graphs with Non-Bipartite Demands}

\hypersetup 
{
 unicode=true,          
 pdfstartview={XYZ null null 1},    
 pdftitle={\articletitle},
 pdfauthor={Lucas Colucci, P\'eter L. Erd\H{o}s, Ervin Gy\H{o}ri, Tam\'as R\'obert Mezei},
 colorlinks=true,        
 allcolors=magenta
}

\newtheorem{theorem}{Theorem}
\newtheorem*{theorem*}{Theorem}
\newtheorem{corollary}[theorem]{Corollary}

\newtheorem{proposition}[theorem]{Proposition}

\newcommand\blfootnote[1]{%
\begingroup
\renewcommand\thefootnote{}\footnote{#1}%
\addtocounter{footnote}{-1}%
\endgroup
}

\makeatletter
\g@addto@macro{\endabstract}{\@setabstract}
\newcommand{\authorfootnotes}{\renewcommand\thefootnote{\@fnsymbol\c@footnote}}%
\makeatother
\begin{document}
\thispagestyle{empty}
\addtocounter{footnote}{1}

\mbox{}
\begin{center}
  \vspace{36pt}

  {\Large {\bfseries\boldmath \articletitle} \\ \medskip \normalsize Edge-disjoint paths in complete bipartite graphs} \par \bigskip

  \blfootnote{E-mail addresses: \texttt{lucas.colucci.souza@gmail.com}, \texttt{erdos.peter@renyi.mta.hu},\\\texttt{gyori.ervin@renyi.mta.hu}, \texttt{tamasrobert.mezei@gmail.com}}
  \authorfootnotes
  \small\scshape
  Lucas Colucci\textsuperscript{2},
  P\' eter L.~Erd\H{o}s\textsuperscript{1,}\footnote{Research  of the author was supported by the National Research, Development and Innovation, NKFIH grant K~116769.}\textsuperscript{,}\footnote{Research  of the author was supported by the National Research, Development and Innovation, NKFIH grant SNN~116095.}, \\
  Ervin Gy\H{o}ri\textsuperscript{1,2,}\addtocounter{footnote}{-2}\footnotemark\textsuperscript{,}\footnotemark,
  Tam\'as R\'obert Mezei\textsuperscript{1,2,}\addtocounter{footnote}{-2}\footnotemark\textsuperscript{,}\addtocounter{footnote}{+1}\footnote{Corresponding author}
  \par \bigskip

  \normalfont
  \textsuperscript{1}Alfréd Rényi Institute of Mathematics, Hungarian Academy of Sciences, Re\'altanoda~u.~13--15, 1053 Budapest, Hungary \par
  \textsuperscript{2}Central European University, Department of Mathematics and its Applications, N\'ador~u.~9, 1051 Budapest, Hungary \par \bigskip

  \normalsize\today
\end{center}

\begin{abstract}
  We investigate the terminal-pairability problem in the case when the base graph is a complete bipartite graph, and the demand graph is a (not necessarily bipartite) multigraph on the same vertex set. In computer science, this problem is known as the edge-disjoint paths problem. We improve the lower bound on the maximum value of $\Delta(D)$ which still guarantees that the demand graph $D$ has a realization in $K_{n,n}$. We also solve the  extremal problem on the number of edges, i.e., we determine the maximum number of edges which guarantees that a demand graph is realizable in $K_{n,n}$.
  
  \bigskip
  
\end{abstract}

\section{Introduction and main results}

The \emph{terminal-pairability} problem has been introduced in~\cite{TPP92}. The basic question is as follows: Let $G$ be a simple graph --- the \emph{base graph}, and a let $D$ be a loopless multigraph with the same vertex set $V(D)=V(G)$  --- the \emph{demand graph}. Can we find a path $P(e)$ for every edge $e \in E(D)$ such that $P(e)$ joins the end-vertices of $e$ and these paths are pairwise edge-disjoint? If there is such a collection of paths then we say that $D$ is \emph{realizable} in $G$. The collection of paths is called a \emph{realization} of $D$ in $G$.

\medskip

This problem has several names in the literature depending on motivation and background. In the terminal-pairability context, sufficient conditions (which guarantee the existence of a realization) and their extremum are sought after. 
In computer science, where the problem is referred to as the \emph{edge-disjoint paths} problem (\textsc{EDP} problem for short), the complexity of constructing the set of edge-disjoint paths is studied (generally both $D$ and $G$ are part of the input). In the following few paragraphs we take a short detour to survey the previous results about the complexity of the \textsc{EDP} problem.

\medskip

The decision version of \textsc{EDP} was first shown to be \textsc{NP}-complete by \citet{MR0471974}.
\citet{MR1309358} proved that for a fixed number of paths the problem is solvable in polynomial time, and the running time was later improved by \cite{MR2885428} (these results are about vertex-disjoint paths, but by moving to the line graph of $G$, edge-disjoint paths become vertex-disjoint).
However, if the number of required paths is part of the input then the problem is \textsc{NP}-complete even for complete (see~\cite{kos}) and series-parallel graphs (see~\cite{Nishi}). The problem is \textsc{NP}-hard even if $G+D$ (the graph obtained by taking the disjoint union of the edge sets) is Eulerian and $D$ consists of at most three set of parallel edges, as shown by \citet{MR1339600}. If no restrictions are made on $G$, then the problem is \textsc{NP}-hard for one set of parallel edges which should be mapped to edge-disjoint paths of length exactly 3, see~\cite{MR2956175}.

\medskip

The edge-disjoint paths problem has many practical applications in telecommunications, VLSI design, network science, see \cite{MR1146902}, for example. Several random and deterministic methods have been developed to solve this problem for special classes of graphs when the number of demand edges is not too high. From a long series of papers we single out the paper of \citet{alon}, where the interested reader can also find a good survey of the earlier developments. 





\begin{theorem*}[Theorem 1.1\ in \cite{alon}]
	Let $G=(V,E)$ be a very strong $d$-regular expander on $n$-vertices. If $D$ is a demand graph with respect to $G$ with at most $\frac{nd}{150\log n}$ edges and maximum degree $d/3$, then there is an online, deterministic, polynomial time algorithm that computes a realization of $D$ in $G$.
\end{theorem*}

Online here means that the algorithm receives the demand edges one-by-one, and designs the paths immediately, without any information on the still forthcoming demand edges.

\medskip

For the special case of complete bipartite base graphs, the upper bound on the number of edges in the demand graph is $\frac{n^2}{75\log(2n)}$. Theorem~\ref{thm:deg1} increases the number of possible demand edges asymptotically by a factor of $\log(2n)$ to $\frac{n^2}{4}$, at the cost of restricting the maximum degree of the demand graph to $(1-o(1))\frac{n}{4}$.

\medskip

The problem is also very closely related to the integer multicommodity flow problems and the theory of graph immersions, each with their own terminologies. In this paper from now on we use the terminology of terminal-pairability, as other papers \cite{TPP92,MR1157508,MR1208923,MR1677781,MR1985088,GyarfasSchelp98,KKGL99,MR3262364,MR3373339,MR3518137,GyMM16Complete} about sufficient conditions do.

\medskip

The terminal-pairability problem arose as a theoretical framework for the practical problem of constructing high throughput packet switching networks. The problem was originally studied by \citet{TPP92}. Their research served as a substrate for further theoretical studies by \citet{GyarfasSchelp98,KKGL99}.

\medskip

In general, it is hopeless to give a condition which is both necessary and sufficient for $D$ to be realizable in $G$, because the problem is \textsc{NP}-hard. Instead, we group the instances of the problem according to the value of a parameter which corresponds to the complexity of it: the maximum degree, or the number of edges of $D$. Given a fixed value of one of these parameters, we are able to give relatively tight conditions which guarantee the existence of a realization. Moreover, for an instance of the problem satisfying these conditions a solution can be constructed in polynomial time.

\medskip

This paper is the latest piece in a series of papers about terminal-pairability \cites{GyMM16Complete,GyMM16Grid,CompleteBip2017}. Our previous paper in the series \cites{CompleteBip2017} also deals with terminal-pairability in complete bipartite base graphs, but with the not very natural restriction that $D$ is bipartite with respect to the same vertex classes as the base graph. The novelty of this paper is that almost the same conditions are sufficient even if the bipartiteness condition on $D$ is omitted.

\medskip

We will refer to an instance of the edge-disjoint paths problem with a pair of graphs $(G,D)$, which implicitly assumes that the underlying sets of vertices of the graphs are identical (or that there is a 1-to-1 correspondence between them).

\medskip

For an edge $e\in E(D)$ with end-vertices $x$ and $y$, we define the \emph{lifting of $e$ to a vertex $z\in V(D)$}, as an operation which transforms $D$ by deleting $e$ and adding a new edge joining the vertices $x$ and $z$ and another new one joining $z$ and $y$; if $z=x$ or $z=y$, the operation does not do anything. We stress that we do not use any information about $G$ to perform a lifting and that the graph obtained using a lifting operation is still a demand graph $D'$. Notice that $D$ is realizable in $G$ a realization if $D'$ is realizable in $G$. Throughout the paper, the demand graphs will be denoted by $D$ and its (indexed) derivatives.

\medskip

It is easy to see that the terminal-pairability problem defined by $G$ and $D$ is solvable if and only if there exists a series of liftings, which, applied successively to $D$, results in a (simple!) subgraph of $G$ (which is a realization of $D$ in $G$). The edge-disjoint paths can be recovered by assigning pairwise different labels to the edges of $D$, and performing the series of liftings so that new edges inherit the label of the edge they replace. Clearly, edges sharing the same label form a walk between the endpoints of the demand edge of the same label in $D$, and so there is also such a path.

\medskip

Several attempts have been made to improve the general result for the case of complete base graphs. One of them is due to \citet{kos}, and an even better extremal bound is proven in~\cite{GyMM16Complete}.

\medskip

In our previous paper~\cite{CompleteBip2017} on terminal-pairability in complete bipartite graphs, we restricted ourselves to demand graphs that are bipartite with respect to the color classes of the base graph.

\medskip

In this paper we explore the terminal-pairability problem when $G$ is a complete symmetric bipartite graph, and the demand graph is assumed to be a loopless multigraph. Our approach is extremal in nature: we are searching for simple conditions on the maximum degree and the number of edges that guarantee realizability.

\medskip

Let $A=\{a_1,\ldots,a_n\}$ and $B=\{b_1,\ldots,b_n\}$. We define $K_{n,n}$ as the graph whose vertex set is $A\cup B$, and its edge set is the set of all unordered $A,B$ pairs. For a (multi-)graph $H$, whose vertex set is $A\cup B$, let $H[A,B]$ be the bipartite (multi-)subgraph of $H$ induced by $A$ and $B$ as the two color classes. For a (multi-)graph $H$, let $\Delta(H)$ be the maximum degree in $H$, and let $e(H)$ be the number of its edges (with multiplicity).

\medskip

A motivation behind taking $K_{n,n}$ as the base graph is that from a bundle of parallel edges between the two color classes, each edge (except at most one) must be mapped to a path of at least 3 edges. If the base graph is a complete graph, multiple edges only need to be mapped to paths of length at least 2, so studying terminal-pairability in complete bipartite graphs is a logical next step for this reason, too.

\medskip

The observation of the previous paragraph and the pigeonhole principle implies that a demand graph containing $\left\lceil\frac{n}{3}\right\rceil+1$ copies of the edge $\{a_i,b_i\}$ for each $i=1,\ldots,n$ is not realizable in $K_{n,n}$:
\[ n\cdot \left(3\cdot\left\lceil\frac{n}{3}\right\rceil+1\right)>e(K_{n,n}).\]
Therefore, the extremal upper bound on the maximum degree of a demand graph realizable in $K_{n,n}$ is at most $\left\lceil\frac{n}{3}\right\rceil$. It is not clear if this is the extremal bound. Our attempts to adapt the proof of \cite{MR3828771} (which improves on the trivial extremal upper bound on the demand degrees in complete graphs) to complete bipartite base graphs have been futile.

\medskip

Although the following theorem is not likely to be sharp, in terms of maximum degree, it approaches the $\left\lceil\frac{n}{3}\right\rceil$ upper bound to a factor of $\frac43$.

\medskip

\begin{theorem}\label{thm:deg1}
  Let $(K_{n,n},D)$ be an instance of the \textsc{EDP} (terminal-pairability) problem, where $D$ is not necessarily bipartite. If
  \[ \Delta(D)\le (1-o(1))\cdot\frac{n}{4} \]
  as $n\to\infty$, then $D$ is realizable in $K_{n,n}$.
\end{theorem}

\medskip

Such sufficient conditions are studied in the theory \emph{immersions} as well. It is a relatively new and quickly growing subject of graph theory. In short, a loopless (multi)graph $D$ has an immersion in a graph $G$ if there exists a mapping of $V(D)$ into $V(G)$ so that $D$ has a realization in $G$ with respect to this vertex-map. The foundations have been laid down by \citet{MR2595703,MR2063516,MR1146902,MR1306142}. Recently, \citet{MR3231086,MR3223965,dvorak2015complete} studied the problem of finding sufficient conditions on a simple graph so that it contains an immersion of $K_n$. Theorem~\ref{thm:deg1} (and even more so, the main result in \cite{GyMM16Complete}) studies the converse of this problem: what is a sufficient condition on a loopless multigraph $D$ so that it has an immersion into $K_{n,n}$ (respectively, $K_n$). In comparison, Theorem~2 in \cite{CompleteBip2017} is not invariant on the permutation of $V(D)$, but the technical novelty in the proofs of this paper (compared to \cite{CompleteBip2017}) allows us to reformulate Theorem~\ref{thm:deg1} in the language of immersions.

\begin{corollary}
If $H$ is a loopless multigraph on at most $2n$ vertices with maximum degree at most $(1-o(1))\cdot\frac{n}{4}$, then there is an immersion of $H$ in $K_{n,n}$.
\end{corollary}

\medskip

If a significant amount of the edges of $D$ are inside the two color classes, Theorem~\ref{thm:deg2} permits even higher degrees in $D$.

\medskip

\begin{theorem}\label{thm:deg2}
  Let $(K_{n,n},D)$ be an instance of the \textsc{EDP} problem. If
  \[ \Delta(D)\le (1-o(1))\cdot\left(\frac{2n}{7}-\frac37\cdot \frac{e(D[A,B])}{n}\right) \]
  as $n\to\infty$, then $D$ is realizable in $K_{n,n}$.
\end{theorem}

Additionally, we prove a sharp bound on the maximum number of edges in a realizable demand graph:

\begin{theorem}\label{thm:edge}
  Let $n\geq 2$ and $(K_{n,n},D)$ be an instance of the \textsc{EDP} problem. If $D$ has at most $2n-3$ edges and $\Delta(D)\le n$, then $D$ is realizable in $K_{n,n}$.
\end{theorem}

The assumption $\Delta(D)\le n$ is trivially necessary: at any given vertex, there can be at most $n$ edge-disjoint paths that terminate there. The result is sharp, as shown by the demand graph on $2n-2$ edges consisting of two bundles of $n-1$ edges, where one of the bundles joins an arbitrary pair of vertices in $A$, while the other bundle joins a pair in $B$.

\medskip

The following \textsc{NP}-hardness is result probably well-known, but we have not been able to find a reference for it. For completeness' sake, we include a short reduction.

\medskip 

\begin{proposition}
	The terminal-pairability or edge-disjoint paths problem for $K_{n,n}$ is \textsc{NP}-hard.
\end{proposition}
\begin{proof}
	The EDP problem is \textsc{NP}-hard on complete base graphs~\cite{kos}.
	Take an instance of the edge-disjoint paths problem $(K_n,D)$ on the vertex set $\{a_1,\dots, a_n\}$. Take $G=K_{\binom{n}{2},\binom{n}{2}}$, such that its vertex classes are $\{a_1,\ldots,a_{\binom{n}{2}}\}$ and $\{b_1,\ldots,b_{\binom{n}{2}}\}$
	Let $(G,D')$ be an instance of the \textsc{EDP} problem, where
	\[ E(D')=E(D)\cup E(G)\setminus \left\{(a_i,b_j),(b_j,a_k)\ :\ 1\le i<k\le n,\  j=i-1+\sum_{l=0}^{k-1}l\right\}. \]
	Obviously, if $D$ has a realization in $K_n$ then we may lift an edge $\{a_i,a_k\}$ of the realization of $D$ to the vertex $b_j$ such that $j=i-1+\sum_{l=0}^{k-1}l$. If $D'$ has a realization in $G$ then it also has such a realization where each edge
	\[ e\in E(G)\setminus \left\{(a_i,b_j),(b_j,a_k)\ :\ 1\le i<k\le n,\  j=i-1+\sum_{l=0}^{k-1}l\right\} \]
	is mapped to a path of one edge, i.e., itself: this can clearly be done, if $e$ is not in the realization of $D'$; if it is, we can simply modify a path $P(f)\ni e$ by replacing $P(f)$ with $P(f)\cup P(e)\setminus \{e\}$ (some cycles may have to be pruned). Such a solution trivially corresponds to a realization of $D$ in $K_n$. The described reduction is polynomial.
\end{proof}

\section{Proofs of the degree versions (Theorem~\ref{thm:deg1}~and~\ref{thm:deg2})}

Before we proceed to prove the theorems, we state several definitions and four well-known results about edge colorings of multigraphs.

\medskip

Let $H$ be a loopless multigraph.
Recall, that the \emph{chromatic index} of $H$ (also known as the \emph{edge chromatic number}), denoted by $\chi'(H)$, is the minimum $k$ such that there is a proper $k$-coloring of the edges of $H$. An \emph{equitable edge coloring} of $H$ is a proper coloring of the edges $E(H)$ such that the sizes of the color classes differ by at most one. The \emph{list chromatic index} of $H$ (also known as the \emph{list edge chromatic number}), denoted by $\mathrm{ch}'(H)$, is the smallest integer $k$ such that if for each edge of $H$ there is a list of $k$ different colors given, then there exists a proper coloring of the edges of $H$ where each edge gets its color from its list.

\medskip

The \emph{maximum multiplicity} $\mu(H)$ is the maximum number of edges joining the same pair of vertices in $H$. The number of edges joining a vertex $x\in V(H)$ to a subset $A\subseteq V(H)$ of vertices is denoted by $e_H(x,A)$. The \emph{set of neighbors} of $x$ in $H$ is denoted by $N_H(x)$. For other notation the reader is referred to~\cite{Diestel}.

\medskip

\begin{proposition}\label{prop:equitable}
  If $H$ is a multigraph and $\chi'(H)\le k$ for some integer $k$, then there is an equitable edge coloring of $H$ with exactly $k$ colors.
\end{proposition}
\begin{proof}
  Let $c:E(H)\to \{1,2,\ldots,k\}$ be a proper edge coloring of $H$. Suppose there are two colors $x$ and $y$ for which $|c^{-1}(x)|\geq |c^{-1}(y)|+2$. The connected components of $H_{x,y}=c^{-1}(x)\cup c^{-1}(y)$ are cycles (where two parallel edges are regarded as a 2-cycle) and paths. In any cycle of $H_{x,y}$, the number of edges of color $x$ is equal to the number of edges of color $y$, therefore $H_{x,y}$ must contain a path component of odd length, with one more edge of color $x$ than of color $y$.
  By switching the two colors in this path, the sum $\sum_{i=1}^k|c^{-1}(i)|^2$ decreases, and we end up with a coloring which is still proper. Thus, if we cannot repeat this procedure anymore, $c$ must be an equitable coloring, as desired.
\end{proof}

We will use the following well-known results about the edge colorings of multigraphs.

\begin{proposition}[Greedy edge coloring]\label{prop:greedy}
	For any multigraph $H$ we have
	\[ \chi'(H)\le\mathrm{ch}'(H)\le 2\Delta(H)-1. \]
\end{proposition}

%

\medskip

\begin{theorem}[\citet{Vizing65}]\label{thm:vizing}
  For any multigraph $H$, its chromatic index
  \[ \chi'(H)\le \Delta(H)+\mu(H). \]
\end{theorem}

\medskip

\begin{theorem}[\citet{Shannon1949}]\label{thm:shannon}
  For any multigraph $H$, its chromatic index
  \[ \chi'(H)\le \frac32\Delta(H). \]
\end{theorem}

\medskip

\begin{theorem}[\citet{Kahn00}]\label{thm:kahn}
  For any multigraph $H$, its list chromatic index
  \[ \mathrm{ch}'(H)\le (1+o(1))\chi'(H). \]
\end{theorem}

\medskip

We need to prove a technical proposition before the proofs of Theorem~\ref{thm:deg1}~and~\ref{thm:deg2}.

\begin{proposition}\label{prop:ABB}
  If $D$ is a demand graph on the vertex set $V(K_{n,n})$ and $\Delta(D)\le n/4$, then there exists a proper edge $2\lfloor n/2\rfloor$-coloring of $D[A,B]\cup D[B]$, which induces an equitable $2\lfloor n/2\rfloor$-coloring on $D[B]$ and an almost equitable (the difference between the sizes of two color classes is $\le 2$) coloring on $D[A,B]$.
\end{proposition}
\begin{proof}
  Observe that (by Proposition~\ref{prop:greedy})
  \begin{align*}
    \chi'(D[B])   & \le 2\Delta(D)\le \frac12n, \\
    \chi'(D[A,B]) & \le 2\Delta(D)\le \frac12n.
  \end{align*}
  By Proposition~\ref{prop:equitable}, there is a partition of $E(D[B])$ into $\lfloor n/2\rfloor $ matchings of size $\lfloor e(D[B])/n\rfloor$ and $\lceil e(D[B])/n\rceil$, say $M_1,\ldots,M_{\lfloor n/2\rfloor}$, so that $|M_i|\ge |M_j|$ for $i<j$.
  Similarly, there is a partition of $E(D[A,B])$ into $\lfloor n/2\rfloor$ matchings of size $\lfloor e(D[A,B])/n\rfloor$ and $\lceil e(D[A,B])/n\rceil$, say $N_1,\ldots,N_{\lfloor n/2\rfloor}$, so that $|N_i|\le |N_j|$ for $i<j$.
  It is sufficient to prove now that for all $i=1,\ldots,\lfloor n/2\rfloor$, there exists a 2-coloring of $M_i\cup N_i$ which is induces an equitable 2-coloring on $D[B]$ and induces an almost equitable 2-coloring on $D[A,B]$.

  \medskip

  Fix $i$. Observe, that $M_i\cup N_i$ is the vertex disjoint union of some edges and paths composed of two or three edges that alternate between elements of $M_i$ and $N_i$. The paths of two and three edges contain one edge of $M_i$ exactly. Let the number of components of $M_i\cup N_i$ containing $k$ edges be $c_k$.
  
  \medskip
  
  Color the $M_i$ edge of $\lfloor c_3/2\rfloor$ of the path components of length three with color~1, and color the $M_i$ edges of the remaining $\lceil c_3/2\rceil$ paths of length three with color~2. Similarly, color the $M_i$ edge of $\lceil c_2/2\rceil$ paths of length two with color~1, and color the remaining  $\lfloor c_2/2\rfloor$ uncolored $M_i$ edges in paths of length two with color~2.
  
  \medskip
  
  The already colored edges in $M_i$ determine the colors of edges of $N_i$ intersecting them (as we are looking for a proper edge 2-coloring). Let this proper partial edge 2-coloring be $c$. It trivially induces an equitable coloring on $D[B]$. On $D[A,B]$, we have $|c^{-1}(1)\cap D[A, B]|=2\lceil c_3/2\rceil+\lfloor c_2/2\rfloor$ and $|c^{-1}(2)\cap D[A, B]|=2\lfloor c_3/2\rfloor+\lceil c_2/2\rceil$, the difference of which is clearly at most 2.
  As the yet uncolored edges of $M_i\cup N_i$ are vertex disjoint, this partial coloring can be extended to a proper 2-coloring, which is equitable in $M_i$ and almost equitable in $N_i$.
\end{proof}

\begin{proof}[\normalfont\bfseries Proof of Theorem~\ref{thm:deg1}]
  As $D$ has an even number of vertices, we may assume that $D$ is regular by adding edges, if necessary. Clearly, $e(D[A])=e(D[B])$, $e(D)=e(D[A])+e(D[A,B])+e(D[B])$, and ${e(D)=n\cdot\Delta(D)}$.

  \medskip

  Our proof consists of three steps. In the first step, we resolve the high multiplicity edges of $D[A]$, while leaving $D[A,B]\cup D[B]$ untouched. In the second step, we lift the edges of $D[B]$ to $A$, and resolve the multiplicities of $D[A,B]$. In the third step, we lift the edges induced by $A$ to $B$, while preserving a simpleness of the bipartite subgraph induced by $A$ and $B$, thus we end up with a graph which is a realization of $D$.

  \medskip

  By Proposition~\ref{prop:greedy}, $\chi'(D[A])\le n$, so Proposition~\ref{prop:equitable} implies the existence of an equitable edge $n$-coloring $c_1$ of $D[A]$.
  We construct $D'$ from $D$ by lifting the elements of $c_1^{-1}(i)$ to $a_i$ for all $i=1,\ldots, n$. As $c_1$ is a proper coloring, $\mu(D'[A])\le 2$.
  For any $a\in A$ and $b\in B$, we have the following estimates:
  \begin{align*}
    e_{D'}(a,A) & \le e_D(a,A) +2\cdot\lceil e(D[A])/n\rceil,              \\
    e_{D'}(a,B) & =e_D(a,B),\ e_{D'}(b,A)=e_D(b,A),\ e_{D'}(b,B)=e_D(b,B).
  \end{align*}

  \medskip

  For the second step, we use Proposition~\ref{prop:ABB} to take a proper edge $n$-coloring $c_2$ of $D[A,B]\cup D[B]$, which is an (almost) equitable $n$- or $(n-1)$-coloring if restricted to both $D[A,B]$ and $D[B]$. We get $D''$ from $D'$ by lifting the elements of $c^{-1}_2(i)$ to $a_i$ for all $i=1,\ldots,n$.
  As $c_2$ is a proper edge coloring, $D''[A,B]$ is simple, and $\mu(D''[A])\le \mu(D'[A])+2\le 4$.
  For any $a\in A$ and $b\in B$, we have the following estimates:
  \begin{align*}
    e_{D''}(a,A) & \le e_{D'}(a,A)+e_{D'}(a,B)+\lceil e(D[A,B])/(n-1) \rceil +1                \\
    e_{D''}(a,B) & \le \lceil e(D[A,B])/(n-1) \rceil+1 + 2\cdot\lceil e(D[B])/(n-1) \rceil \\
    e_{D''}(b,A) & =\Delta(D),\ e_{D''}(b,B)=0.
  \end{align*}

  To each edge $e\in E(D''[A])$ with end vertices $a_i$ and $a_j$, we associate a list $L(e)$ of vertices of $B$, to which we can lift $e$ to without creating multiple edges:
  \begin{align*}
    L(e) & =B\setminus \left(N_{D''}(a_i)\cup N_{D''}(a_j)\right),
  \end{align*}
  whose size is bounded from below
  \begin{align*}
    |L(e)| & \geq n-e_{D''}(a_i,B)-e_{D''}(a_j,B) \geq                                  \\
    & \geq n-2\cdot \lceil e(D[A,B])/(n-1) \rceil - 4\cdot\lceil e(D[B])/(n-1) \rceil-2.
  \end{align*}
  
  \medskip
  
  By \citeauthor{Vizing65}'s theorem (Theorem~\ref{thm:vizing}),
  \begin{align*}
    \chi'(D''[A]) & \leq \Delta(D''[A])+\mu(D''[A]) \leq                                                     \\
    & \le \max_{a\in A}\left(e_{D'}(a,A)+e_{D'}(a,B)+\lceil e(D[A,B])/(n-1) \rceil+1\right)+4\le \\
    & \le \Delta(D)+2\cdot\lceil e(D[A])/n\rceil+\lceil e(D[A,B])/(n-1) \rceil+5.
  \end{align*}
  
  \medskip
  
  By Kahn's theorem (Theorem~\ref{thm:kahn}), $\text{ch}'(D''[A])\leq(1+o(1))\chi'(D''[A])$.
  We have $\mathrm{ch}'(D''[A])\le |L(e)|$ for each edge $e$ in $E(D''[A])$, if
  \begin{align*}
    (1+o(1)) & \left(\Delta(D)+2\cdot\lceil e(D[A])/n\rceil+\lceil e(D[A,B])/(n-1) \rceil\right)\le \\
    & \le n-2\lceil e(D[A,B])/(n-1) \rceil - 4\cdot\lceil e(D[B])/(n-1)\rceil.
  \end{align*}
  This inequality holds, if
  \begin{align*}
    (1+o(1)) & \left(\Delta(D)+2\cdot e(D[A])/n+3\cdot e(D[A,B])/n +4\cdot e(D[B])/n\right)\le n.
  \end{align*}
  Using our observations at the beginning of this proof, the previous inequality is a consequence of the regularity of $D$ and
  \begin{align*}
    & (1+o(1))\cdot 4\cdot\Delta(D)\le n.
  \end{align*}
  Thus, if the conditions of the statement of this theorem hold, there is a proper list edge coloring $c_3$ which maps each $e\in E(D''[A])$ to an element of $L(e)$. Finally, we lift every edge $e\in E(D''[A])$ to $c_3(e)$. As we do not create multiple edges between $A$ and $B$, the resulting graph is a realization of $D$.
\end{proof}

\begin{proof}[\normalfont\bfseries Proof of Theorem~\ref{thm:deg2}]
  This proof is a slight variation on the previous proof. We do not lift edges of $D[A]$ to elements of $A$. Futhermore, instead of Vizing's theorem, Shannon's theorem (Theorem~\ref{thm:shannon}) will be used to bound the chromatic index of a graph induced by $A$.

  \medskip

  We may assume that $D$ is regular. For the first step, we use Proposition~\ref{prop:ABB} to take a proper edge $n$-coloring $c_1$ of $D[A,B]\cup D[B]$, which is an (almost) equitable $n$- or $(n-1)$-coloring if restricted to $D[A,B]$ and $D[B]$. Lift $c^{-1}_1(i)$ to $a_i$ for all $i=1,\ldots,n$ to get $D'$ from $D$. Now $D'[A,B]$ is simple and $D'[B]$ is an empty graph on $n$-vertices.
  For any $a\in A$ and $b\in B$, we have the following estimates:
  \begin{align*}
    e_{D'}(a,A) & \le e_{D}(a,A)+e_{D}(a,B)+\lceil e(D[A,B])/(n-1) \rceil+1,               \\
    e_{D'}(a,B) & \le \lceil e(D[A,B])/(n-1) \rceil+1 + 2\cdot\lceil e(D[B])/(n-1) \rceil, \\
    e_{D'}(b,A) & =\Delta(D),\ e_{D'}(b,B)=0.
  \end{align*}

  To each edge $e\in E(D'[A])$ with end vertices $a_i$ and $a_j$, we associate a list $L(e)$ of vertices of $B$, to which we can lift $e$ to without creating multiple edges:
  \begin{align*}
    L(e) & =B\setminus \left(N_{D'}(a_i)\cup N_{D'}(a_j)\right),
  \end{align*}
  whose size is bounded from below
  \begin{align*}
    |L(e)| & \geq n-e_{D'}(a_i,B)-e_{D'}(a_j,B) \geq                                      \\
    & \geq n-2\lceil e(D[A,B])/(n-1) \rceil - 4\cdot\lceil e(D[B])/(n-1) \rceil-2\ge \\
    & \ge n-(1+o(1))2\Delta(D).
  \end{align*}
  By Shannon's theorem (Theorem~\ref{thm:shannon}),
  \begin{align*}
    \chi'(D'[A]) & \leq \frac32\Delta(D'[A])\leq \\
    &\le \frac32\cdot \max_{a\in A}\left(e_{D}(a,A)+e_{D}(a,B)+\lceil e(D[A,B])/(n-1) \rceil+1\right)\le \\
    & \le (1+o(1))\cdot\frac32\cdot\left(\Delta(D)+e(D[A,B])/n \right).
  \end{align*}
  Furthermore, by Kahn's theorem (Theorem~\ref{thm:kahn}), $\text{ch}'(D'[A])\leq(1+o(1))\chi'(D'[A])$.
  We have $\mathrm{ch}'(D'[A])\le |L(e)|$ for each edge $e$ in $E(D'[A])$, if
  \begin{align*}
    (1+o(1)) & \cdot\frac32\cdot\left(\Delta(D)+ e(D[A,B])/n \right)\le n-2\Delta(D).
  \end{align*}
  This holds, if
  \begin{align*}
    (1+o(1)) & \cdot\left(\frac72\cdot\Delta(D)+\frac32\cdot  \frac{e(D[A,B])}{n}\right)\le n.
  \end{align*}
  Thus, if the conditions of the statement of this theorem hold, there is a proper list edge coloring $c_2$ which maps each $e\in E(D'[A])$ to an element of $L(e)$. Finally, we lift every edge $e\in E(D'[A])$ to $c_2(e)$. As we do not create multiple edges between $A$ and $B$, the resulting graph is a realization of $D$.
\end{proof}

\section{Algorithmic versions of Theorems \ref{thm:deg1} and \ref{thm:deg2} }
\subsection{Analysis of the complexity of the \texorpdfstring{\textsc{EDP}}{EDP} problem with degree conditions}
The proofs of Theorems~\ref{thm:deg1}~and~\ref{thm:deg2} provide recipes for constructing realizations in $K_{n,n}$. In this section we derive randomized and deterministic polynomial algorithms from them.
 
\medskip

\begin{proposition}\label{prop:algeq}
	Given a proper edge $k$-coloring $c:E(H)\to \{1,\ldots,k\}$, an equitable proper edge $k$-coloring of $H$ can be computed in $\mathcal{O}(kn)$.
\end{proposition}
\begin{proof}
	Recall the proof of Proposition~\ref{prop:equitable}. Given two colors $i$ and $j$ such that $|c^{-1}(i)|> \lceil e(H)/k\rceil$ and $|c^{-1}(j)|< \lfloor e(H)/k\rfloor$, find the connected components of $c^{-1}(i)\cup c^{-1}(j)$. By switching colors on path components containing more edges of color $i$ than $j$, we eventually reach a point where either $|c^{-1}(i)|$ has decreased to $\lceil e(H)/k\rceil$ or $|c^{-1}(j)|$ has increased to $\lfloor e(H)/k\rfloor$. This subroutine takes $\mathcal{O}(|c^{-1}(i)|+|c^{-1}(j)|)$ time.
	
	\medskip
	
	Repeat the above procedure until every color class has cardinality $\le \lceil e(H)/k\rceil$ or every color class has size  $\ge \lfloor e(H)/k\rfloor$. The subroutine is called at most $k$ times, so the algorithm took $\mathcal{O}(kn)$ time. To make the coloring equitable, at most $k$ more color switches have to be performed, each taking at most $\mathcal{O}(n)$ time.
\end{proof}
Consequently, the algorithm in the proof of Proposition~\ref{prop:ABB} also runs in $\mathcal{O}(kn)$ time.

\medskip

The proof of \citeauthor{Kahn00}'s result \cite{Kahn00} is probabilistic, and it can be emulated in polynomial time. It immediately follows that:

\begin{theorem}\label{thm:cplx1}
	Given an instance $(K_{n,n},D)$ of the \textsc{EDP} problem, where $\Delta(D)\le (1-o(1))\cdot\frac{n}{4}$, there is a randomized polynomial time algorithm which computes a realization of $D$ in $K_{n,n}$.
\end{theorem}

By using greedy edge coloring algorithms, we get deterministic algorithms (albeit with tighter upper bounds on $\Delta(D)$).
\begin{theorem}\label{thm:cplx2}
	Given an instance $(K_{n,n},D)$ of the \textsc{EDP} problem, where 
\[ \Delta(D)\le \frac16 (n-7),\text{ or}\]
\[ \Delta(D)\le \frac14\left( n-2\cdot\left\lceil\frac{e(D[A,B])}{n-1} \right\rceil-5\right), \]
there is a deterministic $\mathcal{O}(\Delta(D)\cdot n)$ time algorithm which computes a realization of $D$ in $K_{n,n}$.
\end{theorem}
\begin{proof}
	We make the demand graph regular in $\mathcal{O}(\Delta(D)\cdot n)$. 
	The time complexity of the greedy edge coloring algorithm is linear, and emulating Proposition~\ref{prop:ABB} requires $\mathcal{O}(\Delta(D)\cdot n)$ time (see Proposition~\ref{prop:algeq}). Constructing $D'$ and $D''$ by lifting the appropriate edges only takes linear time.
	
	\medskip
	
	In the proof of Theorem~\ref{thm:deg1}, replace \citeauthor{Kahn00}'s list edge coloring with the greedy algorithm (Proposition~\ref{prop:greedy}). To get a proper list edge coloring of $D''[A]$, it is sufficient to have
	\begin{align*}
		2\Delta(D''[A])-1\le n-2\cdot \left\lceil \frac{e(D[A,B])}{n-1} \right\rceil - 4\cdot\left\lceil \frac{e(D[B])}{n-1} \right\rceil-2 \\
		2\Delta(D)+4\cdot\left\lceil \frac{e(D[A])}{n}\right\rceil+4\cdot \left\lceil \frac{e(D[A,B])}{n-1} \right\rceil+4\cdot\left\lceil \frac{e(D[B])}{n-1}\right\rceil\le n-3
	\end{align*}
	which is satisfied if $\Delta(D)\le \frac16(n-7)$.
	
	\medskip
	
	In the proof of Theorem~\ref{thm:deg2}, to get a proper list edge coloring of $D'[A]$ via greedy coloring, it is sufficient to have
	\begin{align*}
		2\cdot \left(\Delta(D)+\left\lceil \frac{e(D[A,B])}{n-1}\right\rceil+1\right)-1\le n-2\cdot \left\lceil \frac{e(D[A,B])}{n-1} \right\rceil - 4\cdot\left\lceil \frac{e(D[B])}{n-1} \right\rceil-2 \\
		2\cdot \Delta(D)+4\cdot\left\lceil \frac{e(D[A,B])}{n-1}\right\rceil+4\cdot\left\lceil \frac{e(D[B])}{n-1} \right\rceil\le n-3
	\end{align*}
	which is satisfied when the second inequality of the conditions of this theorem holds.
\end{proof}

\medskip

These bounds are not as tight as our theoretical bounds, but are smaller only by a factor of $\frac32$ and $\frac87$, respectively. 

\subsection{Approximate solutions to the \textsc{MaxEDP} problem}

Given an instance $(G,D)$ of the \textsc{EDP} problem, the \textsc{MaxEDP} problem asks for the subgraph $D'\subseteq D$ with the maximum number of edges such that $(G,D')$ is realizable.

\medskip

The algorithms for \textsc{EDP} in the previous section can be turned into constant factor approximation algorithms for the \textsc{MaxEDP} problem at the cost of a sensible additional term to their running time.


\begin{theorem}\label{thm:approxMaxEDP}
	Let $(K_{n,n},D)$ be an instance of the \textsc{MaxEDP} problem. Let $\mathcal{A}$ be an algorithm that can solve an instance $(K_{n,n},D^*)$ of the \textsc{EDP} problem given a maximum degree condition $\Delta(D^*)\le t$. Then there is an algorithm $\mathcal{B}$ which gives a $(3n/2t)$-approximation solution to the \textsc{MaxEDP} problem. The running time of algorithm~$\mathcal{B}$ is at most $\mathcal{O}(e(D)\cdot n)$ plus the running time of algorithm $\mathcal{A}$ on the instance $(K_{n,n},D^*)$, where $D^{*}$ is some subgraph of $D$.
\end{theorem}
\begin{proof}
	Let $D_\text{opt}$ be a subgraph of $D$ that has the maximum number of edges among those that are realizable in $K_{n,n}$. Trivially, $\Delta(D_\text{opt})\le n$. Since $D$ is bipartite, using a folklore reduction to a maximum flow algorithm, the subgraph $D^{*}\subseteq D$ with the maximum number of edges such that $\Delta(D')\le t$ can be found in $\mathcal{O}(e(D)\cdot n)$. (For regular $D$, the time complexity can be improved to $\mathcal{O}(e(D)\cdot\log\Delta(D))$ using the algorithm of \cite{MR1805711}.)
	
	\medskip
	
	Partition $E(D_\text{opt})$ into $\lfloor\frac32\Delta(D)\rfloor$ matchings (see Theorem~\ref{thm:shannon}). Order the matchings in decreasing order of their cardinality. Let the spanning subgraph of $D_\text{opt}$ formed by the union of the first $t$ of its largest matchings be $D'$.
	Since $\Delta(D')\le t$, we have $e(D')\le e(D^*)$. Furthermore, because we take the largest matchings,
	\[ |E(D^*)|\ge |E(D')|\ge \frac{2t}{3n}\cdot |E(D^\text{opt})|.\]
	
	\medskip
	
	Lastly, Algorithm~$\mathcal{A}$ can be used to compute a realization of $D^*$ in $K_{n,n}$.
\end{proof}
Using Theorems~\ref{thm:cplx1} and \ref{thm:cplx2}, we have the following corollaries.
\begin{corollary}
	Given an instance $(K_{n,n},D)$ of $\textsc{MaxEDP}$, there is a randomized polynomial time algorithm which produces a $(6+\mathcal{O}(1/n))$-approximation solution.
\end{corollary}
\begin{corollary}
	Given an instance $(K_{n,n},D)$ of $\textsc{MaxEDP}$, there is a deterministic $\mathcal{O}(e(D)\cdot n)$ time algorithm which produces a $(9+\mathcal{O}(1/n))$-approximation solution.
\end{corollary}

\section{Proof of the edge version (Theorem~\ref{thm:edge})}

We apply induction on $n$. It is easy to check the result for $n\leq 3$, so let us assume from now on that $n\geq4$. Since a subgraph of a realizable graph is realizable as well,
it is enough to prove the result for demand graphs $D$ on exactly $2n-3$ edges.
 Recall that $A$ and $B$ be are the color classes of $K_{n,n}$, and let
\[ S=\{v\in A\cup B:d_D(v)\geq n-1\}. \]

\medskip

Since $D$ has $2n-3$ edges, it is clear that $|S|\leq 3$ and that for every pair of vertices in $S$ there is at least one edge joining them.

\medskip

For a vertex $v\in V(D)$, we denote by $d(v)$ its degree and by $\gamma_A(v)$, $\gamma_B(v)$ the number of neighbors of $v$ in class $A$ and $B$, respectively. Let $d'(v), \gamma_A'(v), \gamma_B'(v)$ denote the value of these quantities after resolution of a vertex in $D$; similarly, $d''(v)$, $\gamma_A''(v)$, $\gamma_B''(v)$ denotes the values after the resolution of a second vertex, and so on.  We denote the multiplicity of an edge $uv$ by $\mu(uv)$, and we call it \emph{monochromatic} if $u$ and $v$ are in the same color class of $D$, and \emph{crossing}, otherwise.

\medskip

Notice that, for a vertex $v\in A$, we need precisely $d(v)-\gamma_B(v)$ vertices in $B\backslash N_B(v)$  (which can be freely chosen in this set) to lift all the multiple edges and monochromatic edges incident to $v$. After these liftings, which increased the number of edges of the graph by $d(v)-\gamma_B(v)$, all the edges incident to $v$ have their other endpoint in $B$ and are simple. Clearly, we have the same for a vertex in $B$, exchanging all the occurrences of $A$ and $B$. We say in this case that $v$ is \emph{resolved}.

\medskip

For the induction step, we will resolve $t=1 \text{ or } 3$ vertices in each color class of $D$ (possibly making some liftings before), remove them from the graph, getting a smaller graph $D'$, and apply the induction hypothesis on $D'$. It is clear that $D$ is realizable if $D'$ is. By the inductive hypothesis, $D'$ is realizable if the following conditions hold:

\medskip

\vbox{\begin{enumerate}
\item $\Delta(D')\leq n-t$,
\item $D'$ has at most $2(n-t)-3$ edges, i.e., there were at least $2t$ edges incident to the $2t$ removed vertices after their resolution.
\end{enumerate}}

\medskip

Assume first that there are $3$ vertices of degree $n$ in $D$ lying on the same color class (this can only happen if $n\geq6$, since we must have $3n\leq \sum_{v\in D} d(v)=4n-6$). 
In this case, all other vertices in $D$ have degree at most $4n-6-3n=n-6$.

\medskip

Let $x,y,z\in A$ be the vertices of degree $n$. As $e(D)=2n-3$, it is clear that we have $\mu(xy)+\mu(xz)+\mu(yz)\geq n+3$ and that there are at least $6$ isolated vertices in $B$. We choose three from them, say, $a,b,c$. Without loss of generality, we may assume that $\mu(xy)+\mu(xz)\geq 2/3\cdot(n+3)\geq 6$.

\medskip

We resolve $x$, $y$ and $z$ in this order. After resolving $x$, we have $\gamma_B'(y)\geq\mu(xy)$, and after resolving $y$, we have $\gamma_B''(z)\geq\mu(xz)$. In total, we add $d(x)-\gamma_B(x)+d'(y)-\gamma_B'(y)+d''(z)-\gamma_B''(z)$ edges to $D$, and we delete at least $d(x)+d'(y)+d''(z)$ edges when we remove $x,y,z,a,b,c$ from $D$, so $e(D')\leq e(D)-(\gamma_B(x)+\gamma_B'(y)+\gamma_B''(z))\leq e(d)-(\mu(xy)+\mu(xz))\leq e(D)-6$ edges, and we can apply induction since $\Delta(D')\leq n-3$.

\medskip

From now on, we may assume that there are at most two vertices of degree $n$ in a class.

\medskip

Let $u$ be a maximum degree vertex in $D$. We may assume that $u\in A$. We distinguish some cases based on the value of $\gamma_B(u)$:

\makeatletter
\renewcommand{\thesubsection}{\bfseries Case~\arabic{subsection}}
\renewcommand\subsection{\@startsection{subsection}{2}{\z@}%
{9\p@ \@plus 6\p@ \@minus 3\p@}%
{3\p@ \@plus 6\p@ \@minus 3\p@}%
{\normalfont\normalsize}}

\renewcommand{\thesubsubsection}{\bfseries Case~\arabic{subsection}.\arabic{subsubsection}}
\renewcommand\subsubsection{\@startsection{subsubsection}{2}{\z@}%
{9\p@ \@plus 6\p@ \@minus 3\p@}%
{3\p@ \@plus 6\p@ \@minus 3\p@}%
{\normalfont\normalsize}}

\renewcommand{\theparagraph}{\bfseries Case~\arabic{subsection}.\arabic{subsubsection}.\arabic{paragraph}}
\renewcommand\paragraph{\@startsection{paragraph}{2}{\z@}%
{9\p@ \@plus 6\p@ \@minus 3\p@}%
{3\p@ \@plus 6\p@ \@minus 3\p@}%
{\normalfont\normalsize}}

\makeatother

\subsection{\texorpdfstring{$\gamma_B(u)\geq2$, or $\gamma_B(u)=1$ and $N_A(u)\neq\emptyset$.}{��\_B(u)\textge 2, or ��\_B(u)=1 and N\_A(u)≠∅.}} We resolve the demands of $u$ first. Then, if $N_A(u) \ne \emptyset$, let $u'\in N_A(u)$ be a vertex of maximum degree in this set, and $v\in B$ be a vertex that was used for a lifting of an edge $uu'$. Otherwise, if $N_A(u)=\emptyset$, let $v$ be an arbitrary neighbor of $u$ in $B$. We resolve the vertex $v$ using the available vertices in $A$ for the lifts in increasing order of degree, and then delete $u$ and $v$. 

\medskip

We claim that the remaining graph $D'$ satisfies $\Delta(D')\leq n-1$. Indeed, the procedure above increases the degree of a vertex by one if it was used for a lift of either $u$ or $v$, does not increase the degree of any other vertex in the graph, and decreases the degree of the neighbors of $u$ in $B$ by at least one. Since the vertices used for a lift of $u$ are not joined to it, and hence have degree at most $n-2$ (recall that every vertex of degree at least $n-1$ is joined to a maximum degree vertex), no vertex in $B$ has degree more than $n-1$ after the procedure. On the other hand, we could have a non-neighbor of $v$, $x\in A$, distinct from $u$ and $u'$, which has degree at least $n-1$ originally. This vertex would end up with degree at least $n$ after the procedure in case it is used for a lift of $v$. The way we chose the vertices in $A$ for lifts of $v$ would imply, however, that $d(v)\leq n-2$, and so $4n-6=\sum_{v \in D}d(v)\geq d(u)+d(u')+d(x)+d(v)=4n-5$, a contradiction.

\medskip

The liftings added $d(u)-\gamma_B(u)+d'(v)-\gamma_A'(v)$ edges to $D$, and we deleted $d(u)+d'(v)-1$ edges when we remove $u$ and $v$, so $e(D')\leq e(D)-(\gamma_B(u)+\gamma_A'(v)-1)\leq e(D)-2$, so we can apply the induction hypothesis on $D'$.

\subsection{\texorpdfstring{$\gamma_B(u)=1$ and $N_A(u)=\emptyset$.}{��\_B(u)=1 and N\_A(u)=∅.}}
Let $u'$ be the neighbor of $u$ in $B$. If $u'$ has another neighbor distinct from $u$, we would have $d(u')>d(u)$, a contradiction. So $uu'$ forms a bundle. Also, if there is any crossing edge $vv'$ not belonging to this bundle, we resolve $u$ first and then $v'\in B$ without using $u$ in a lift (which is possible since we need $d'(v')-\gamma'_A(v')\leq n-2-1=n-3$ vertices of $A$ for the lifts). We are done by induction again after we deleting $u$ and $v$, since $e(D')\leq e(D)-2$ and $\Delta(D')\leq n-1$.

\medskip

Assume now that $E(D)$ consists of the bundle $uu'$ and monochromatic edges not incident to $u$ or $u'$.
In this case, we take $a\neq u$ in $A$, $b\neq u'$ in $B$ with smallest degree (by the number  of edges, it is at most $3$). Let $e$ be an edge, say, in $A$, which is not incident to $a$. We lift $e$ to $b$, and replace one copy of the edge $uu'$ by the path $ubau'$. Then we resolve the multiple edges of $a$ and $b$, and delete both of them. The remaining graph $D'$ has $\Delta(D')\leq n-1$ and two less edges than $D$, so we may apply the induction hypothesis on $D'$.

\subsection{\texorpdfstring{$\gamma_B(u)=0$.}{��\_B(u)=0.}}

Among the neighbors of $u$, let $u'$ be one with the largest degree. Let us consider two cases:

\medskip

\subsubsection{There is an edge $e$ independent of $uu'$.} If $e$ is a crossing edge, let $e=ab$. If not, let $a$ and $b$ be vertices in $A$ and $B$, respectively, distinct from $u$, $u'$ and the endpoints of $e$. In the first case, we lift $uu'$ to $b$, and in the second, we also lift $e$ to the vertex $a$ or $b$ which is in the opposite class of $e$. Then, we resolve the vertices $a$ and $b$ and delete them. In both cases, it is clear that the remaining graph $D'$ satisfies $e(D')\leq e(D)-2$ and $\Delta(D')\leq n-1$, so the result follows from induction applied in $D'$.

\medskip

\subsubsection{There is no edge independent from $uu'$.} As $e(D)=2n-3$ and $d(u), d(u')\leq n$, it follows that $uu'$ is an edge of multiplicity at most 3. So, it is clear that there are two independent edges $e$ and $f$ such that $u$ and $u'$ are incident to $e$ and $f$, respectively. Again, we let $a$, $b$ be vertices in $A$ and $B$ not incident to $e$ or $f$, and we lift both edges to $b$. After resolving and deleting $a$ and $b$, we are left with $D'$ with $\Delta(D')\leq n-1$ and $e(D')\leq e(D)-2$, so we are done by the induction hypothesis on $D'$.

\bigskip

\renewcommand*{\bibfont}{\small}
\section*{References}
\printbibliography[heading=none]

\end{document}